\documentclass[12pt]{amsart}

\usepackage{amsmath,amssymb,amsthm}    
\usepackage[colorlinks]{hyperref} 
\usepackage[abbrev,msc-links]{amsrefs}  
\usepackage{enumerate}
\usepackage{multirow}
\usepackage{enumitem}
\usepackage{float}
\usepackage{mathrsfs}
\usepackage{caption}
\usepackage{comment}

\usepackage[capitalise]{cleveref}

\newcommand{\E}{{\mathbb{E}}}

\newtheorem{theorem}{Theorem}

\newtheorem{lemma}[theorem]{Lemma}

\newtheorem{conjecture}{Conjecture}
\newtheorem{openquestion}[conjecture]{Open question}

\newtheorem{proposition}[theorem]{Proposition}

\theoremstyle{remark}
\newtheorem{remark}[theorem]{Remark}

\newcommand{\blue}[1]{{\textcolor{blue}{#1}}}

\newcommand{\Z}{\mathbb{Z}}
\newcommand{\rr}{\mathbb{E}}

\newcommand{\EE}{\mathscr{E}}

\newcommand{\ball}{\mathcal{B}}

\newcommand{\intr}{{\rm int}\,}
\newcommand{\dist}{{\rm dist}\,}

\newcommand{\wt}{\widetilde}

\title[Upper bounds on chromatic number in low dimensions]{Upper bounds on chromatic number of $\E^n$ in low dimensions}
\author{A. Arman}
\address{Department of Mathematics, University of Manitoba, Winnipeg, MB, R3T 2N2, Canada}
\email{andrew0arman@gmail.com}
\thanks{The first author was supported by NSERC of Canada Discovery Grant RGPIN-2020-05357 and University of Manitoba Research Grant Program project ``Covering problems with applications in Euclidean Ramsey theory and convex geometry''}

\author{A.\ Bondarenko}
\address{Department of Mathematical Sciences, Norwegian University of Science and Technology, NO-7491 Trondheim, Norway}
\email{andriybond@gmail.com}
\thanks{The second author was supported in part by Grant 275113 of the Research Council of Norway.}

\author{A.\ Prymak}
\address{Department of Mathematics, University of Manitoba, Winnipeg, MB, R3T 2N2, Canada}
\email{prymak@gmail.com}
\thanks{The third author was supported by NSERC of Canada Discovery Grant RGPIN-2020-05357.}

\author{D.\ Radchenko}
\address{ETH Zurich, Mathematics Department, Zurich 8092, Switzerland}
\email{danradchenko@gmail.com}

\date{\today}

\begin{document}

	\begin{abstract}
		Let $\chi(\E^n)$ denote the chromatic number of the Euclidean space $\E^n$, i.e., the smallest number of colors that can be used to color $\E^n$ so that no two points unit distance apart are of the same color. We present explicit constructions of colorings of $\E^n$ based on sublattice coloring schemes that establish the following new bounds: $\chi(\E^5)\le 140$, $\chi(\E^n)\le 7^{n/2}$ for $n\in\{6,8,24\}$, $\chi(\E^7)\le 1372$, $\chi(\E^{9})\leq 17253$, and $\chi(\E^n)\le 3^n$ for all $n\le 38$ and $n=48,49$. 
	\end{abstract}	

	\maketitle

	\section{Introduction}
 
	We denote by $\chi(A)$ the chromatic number of $A\subset \E^{n}$, which is the least number of colors needed to color $A$ so that any two points distance one apart receive different colors. Determining $\chi(\E^n)$ is a very challenging question, which is solved only for the trivial case $n=1$, where $\chi(\E^1)=2$. For $n=2$ this problem is known as Hadwiger-Nelson problem. Using a hexagonal tiling, it is not hard to show $\chi(\E^2)\le 7$, and the current best lower bound $\chi(\E^2)\ge 5$ was obtained only relatively recently by de Grey~\cite{deGrey}. 

	There has been a considerable attention paid to the lower estimates of $\chi(\E^n)$ for small $n$, see, e.g.~\cite{RPDG}*{Sect.~5.9}, \cite{Polymath-wiki} or \cite{BoRa}, including the references therein. However, it appears that the upper estimates are almost not studied. In this work we establish a number of new upper bounds on $\chi(\E^{n})$ and hope to encourage further research in this direction. 

	Asymptotically as $n\to\infty$ the best known bounds on $\chi(\E^n)$ are
	\[
	(1.239+o(1))^n\le \chi(\E^n)\le (3+o(1))^n,
	\]
	obtained by Raigorodskii~\cite{raigorod} and Larman and Rogers~\cite{LarRog}, respectively, where the $o(1)$ terms are not quantified. 
	
	Let us summarize the previously best known upper bounds on $\chi(\E^n)$ for small dimensions. 
	The bound $\chi(\E^3)\le 15$ was established by Coulson~\cite{Coulson} and, independently, by Radoi\v ci\'c and T\'oth~\cite{RadToth}. Both constructions were based on coloring the interiors of Voronoi cells of a certain lattice with the same color (for a brief background on lattices the reader may consult \cref{sec:prel}). The choice of the color, however, was performed differently: in~\cite{Coulson}, coset membership w.r.t. a sublattice was used (sublattice coloring scheme), while in~\cite{RadToth} a linear mapping from the lattice to $\Z_{15}$ was used (linear coloring scheme). Extensions to $n=4$ were mentioned without proofs: $\chi(\E^4)\le 54$ in~\cite{RadToth}, and $\chi(\E^4)\le 49$ in~\cite{CoulsonPayne}. We prove that $\chi(\E^4)\le 49$ in \cref{thm:evendims}. The technique of Radoi\v ci\'c and T\'oth~\cite{RadToth}, i.e. linear coloring schemes based on $A_n^*$ lattice, was studied with computer assistance for the cases $n=5$ and $n=6$ by the contributors ``ag24ag24'' (de Grey) and ``Philip Gibbs'' in Polymath~16 project discussion~\cite{Po1}. Namely, it was announced that $\chi(\mathbb{E}^5)\le 156$ and $\chi(\mathbb{E}^6)\le 448$. 
	
	Using computer assistance, we found sublattice coloring schemes for the lattices $A_5^*$, $E_7^*$ and $A_9^*$ yielding the following result.
	\begin{theorem}\label{thm:An_star}
		$\chi(\E^5)\le 140$, $\chi(\E^7)\le 1372$, $\chi(\E^{9})\le 17253$.
	\end{theorem}

	Our next result relies on the following fact: if there exists a lattice $\Lambda$ in $\E^n$ with covering-packing ratio not exceeding $2$, then a  sublattice coloring scheme (with a sublattice $3\Lambda$) gives $\chi(\E^{n})\leq 3^n$. Using known results on laminated lattices and a certain inequality estimating the covering radius, we obtained the following.
	\begin{theorem}\label{thm:3n}
		$\chi(\E^n)\le 3^n$ for any $ n\le 38$ and for $n=48,49$.
	\end{theorem}

	For even dimensions where there exists a lattice with covering-packing ratio at most $3/2$, and some additional properties are known, we can do much better using an appropriate representation of such a lattice as an Eisenstein lattice. This approach also recovers the known bounds for $n=2$ and $n=4$ in a unified way.
	
	\begin{theorem}\label{thm:evendims}
		$\chi(\E^n)\le 7^{n/2}$ for $n\in\{2,4,6,8,24\}$.
	\end{theorem}

	We remark that the proofs of \cref{thm:3n,thm:evendims} do not require computer assistance.

	The current best known upper bounds on $\chi(\E^n)$ for small $n$ are summarized in \cref{tbl:bounds}. Note that the gap between the lower (see, e.g.~\cite{Polymath-wiki}*{Section ``Best known results for the chromatic number in higher dimensions''} or~\cite{BoRa}*{Table~1}) and the upper bounds on chromatic number is large and grows very quickly, namely, for the first few $n$ the known estimates are: $6\le \chi(\E^3)\le 15$, $9\le \chi(\E^4)\le 49$, $9\le \chi(\E^5)\le 140$, $12\le \chi(\E^6)\le 343$, $16\le \chi(\E^7)\le 1372$, $19\le \chi(\E^8)\le 2401$. 

	\begin{table}[h!]
		\begin{center}
			\begin{tabular}{|c|c|c|}
				\hline
				$n$ &  $\chi(\E^n)\le$ & references \\
				\hline
				$2$ & $7$ & John Isbell, see \cite{Soifer}*{p. 29}  \\
				\hline
				$3$ & $15$ & \cite{Coulson}, \cite{RadToth} \\
				\hline
				$4$ & $49$ & \begin{tabular}{@{}c@{}} mentioned in~\cite{CoulsonPayne}; \\ proof in~\cref{thm:evendims}   \end{tabular} \\
				\hline
				$5$ & $140$ &  {\cref{thm:An_star}} \\
				\hline
				$6$ & $343$ &  \cref{thm:evendims}\\
				\hline
				$7$ & $1372$ & \cref{thm:An_star} \\
				\hline
				$8$ & $2401$ & \cref{thm:evendims} \\
				\hline
				$9$ & $17253$ & \cref{thm:An_star} \\
				\hline
				$10\le n\le 21$ & $3^n$ & \cref{thm:3n} \\
				\hline
				$22\le n\le 24$ & $7^{12}$ & \cref{thm:evendims} \\
				\hline
				$25\le n\le 38$, $n=48, 49$ & $3^n$ & \cref{thm:3n} \\
				\hline
			\end{tabular}
		\end{center}
		\caption{Upper bounds on $\chi(\E^n)$ for small $n$}
		\label{tbl:bounds}
	\end{table}

	We give the necessary preliminaries in the next section. \cref{thm:An_star} is proved in \cref{sec:computer}, \cref{thm:3n,thm:evendims} are proved in \cref{sec:covering-packing}. The concluding \cref{sec:questions} lists some open questions.

	\section{Preliminaries}\label{sec:prel}
	
	For $x,y\in\E^n$, we denote by $x\cdot y$ the dot product in $\E^n$ and by $|x|:=(x\cdot x)^{1/2}$ the Euclidean norm of $x$. For two sets $A,B\subset \E^n$ we set $\dist(A,B):=\inf\{|x-y|:x\in A,y\in B\}$ and extend the notation to the one-point sets by $\dist(A,b):=\dist(A,\{b\})$. We also define $\ball(R):=\{x\in\E^n:|x|<R\}$ and $\ball[R]:=\{x\in\E^n:|x|\le R\}$.
	
	\subsection{Lattices} For a real $m\times n$ matrix $M$ of rank $n$, a \emph{lattice} $\Lambda$ generated by $M$ is the set $\Lambda=M\Z^n$. The \emph{rank} of a lattice $\Lambda$ is $n$, and if $M$ is a full rank square matrix, then $\Lambda$ is said to be \emph{full rank}.

	A lattice $\Lambda\subset \rr^{m}$ forms a discrete additive subgroup in $\rr^{m}$. An additive subgroup of $\Lambda$ is called a sublattice. It is not hard to see that $\Lambda^{\prime}$ is a sublattice of $\Lambda$ generated by $M$ if and only if $\Lambda^\prime=MC\Z^n$ for some $m\times m$ integer matrix $C$. \emph{Index} of $\Lambda^{\prime}$ (with respect to $\Lambda$) is an index of $\Lambda^{\prime}$ as a subgroup of $\Lambda$. We use the notation $|\Lambda / \Lambda'|$  for the index and note that it is equal to $|\det C|$. 

	For a lattice $\Lambda$ generated by a matrix $M$, a \emph{fundamental parallelepiped} $P$ is an image of $[0,1]^{n}$ under $M:\rr^{n}\to \rr^{m}$. If $M$ is a full rank square matrix, then $\Lambda+P$ tessellates $\rr^{n}$.

	\emph{Voronoi cell} of a lattice $\Lambda$ around the origin is the set 
	$$V=V(\Lambda)=\{{ x}\in \rr^{m}: |{ x}|\leq |{ x}-{ z}| \text{ for all } z\in \Lambda\}.$$
	If $M$ is full rank square matrix, then the Voronoi cell $V=V(\Lambda)$ is a convex centrally-symmetric body, such that $\Lambda+V$ tessellates $\rr^{m}$.  Since the inequalities $x\cdot z\le |z|^2/2$ and $|x|\le|x-z|$ are equivalent, $V$ can be represented as the intersection of half spaces
	\begin{equation}\label{eqn:voronoi cell prel}
		V = \bigcap_{z\in \Lambda\setminus\{0\}}\{x\in\E^m:x\cdot z\le |z|^2/2\}.
	\end{equation} 


	The \emph{norm} of the lattice $\Lambda$ is the length of the shortest nonzero vector in $\Lambda$.

	Let $\Lambda$ be a full rank lattice in $\rr^{n}$. The \emph{packing radius} of $\Lambda$ is the largest $r$ such that $\Lambda+\ball(r)$ is a disjoint union of balls in $\rr^{n}$, i.e. balls $\ball(r)$ centered at points of $\Lambda$ pack into $\rr^{n}$. The \emph{covering radius} of $\Lambda$ is the minimum over all $R$ such that $\Lambda+\ball[R]=\rr^{n}$, i.e. the balls $\ball[R]$ centered at points of $\Lambda$ cover entire $\rr^{n}$. The ratio between the covering and the packing radii of $\Lambda$ is called the \emph{covering-packing ratio} of $\Lambda$.

	If $V$ is the Voronoi cell of a full rank $\Lambda$, then packing radius of $\Lambda$ is the largest radius of a ball that can be inscribed in $V$ and covering radius is the smallest radius of a ball that contains $V$. Packing radius of any lattice $\Lambda$ is equal to the half of the norm of $\Lambda$. 
	
	An interested reader is referred to \cite{ConSloane} for further background on lattices.


	\subsection{Coloring almost all of the space}

	First we prove that sets of measure zero have no effect on chromatic number $\chi(\E^{n})$. 
	\begin{proposition}\label{prop:measure zero}
		Let $A$ be a set of Lebesgue measure zero in $\rr^{n}$, then  $\chi(\rr^{n})=\chi(\rr^n\setminus A)$.
	\end{proposition}
	\begin{proof}
		Assume the contrary, namely that $\chi(\rr^{n})>\chi(\rr^{n}\setminus A)$. According to de Bruijn--Erd\"{o}s theorem~\cite{deBruijnErdos}, there exists a finite geometric graph $G$ (edges are pairs of points distance $1$ apart), such that $\chi(G)=\chi(\rr^{n})$. Fix a geometric embedding of the set of vertices of $G$ into $\rr^{n}$, to which we will simply refer as $G$. Let $c$ be a proper coloring of $\rr^{n}\setminus A$ into $\chi(\rr^{n}\setminus A)$ colors. Then for any ${ x}\in \rr^{n}$, ${ x}+G$ contains a point of $A$, since otherwise $G$ can be properly colored according to $c$. In particular, for every ${ x}\in \ball(1)$ there is some ${ v}\in G$, such that ${x}+{v}\in A$, which implies that ${x}\in ({v}+\ball(1))\cap A-{v}$. Consequently, 
		$$\ball(1)\subset \bigcup_{{v}\in G} (({v}+\ball(1))\cap A-{v}).$$
		Now, the measure of $\ball(1)$ is nonzero and the measure of each $({v}+\ball(1))\cap A-{v}$ is zero. Since $G$ is finite, we derive a contradiction with the assumption that $\chi(\rr^{n})>\chi(\rr^{n}\setminus A)$.
	\end{proof}

	\subsection{Sublattice coloring schemes}

	Let $\Lambda$ be a full rank lattice in $\mathbb{E}^{n}$, $\Lambda'$ be a sublattice of $\Lambda$, $V$ be the Voronoi cell of $\Lambda$ about the origin and $\intr(V)$ be the interior of $V$. A \emph{sublattice coloring} $c(\Lambda,\Lambda')$ is a coloring of almost all of $\E^n$ into $|\Lambda/\Lambda'|$ colors in which each point of $v+\intr(V)$, where $v\in\Lambda$, is colored according to the equivalence class of $v$ in $\Lambda/\Lambda'$. Note that $\cup_{v\in \Lambda}(v+V)$ tessellates $\rr^{n}$, so the uncolored set $\cup_{v\in\Lambda}(v+\partial V)$, where $\partial V$ is the boundary of $V$, has Lebesgue measure zero in $\E^n$, which makes \cref{prop:measure zero} applicable for sublattice colorings. More precisely, if we define the sublattice coloring chromatic number as
	\[
	\chi_s(\E^n,\Lambda,\Lambda',\ell):=\begin{cases}
		|\Lambda/\Lambda'|, & \text{if there are no monochromatic points in $c(\Lambda,\Lambda')$ distance $\ell$ apart},\\
		\infty, & \text{otherwise},
	\end{cases}
	\]
	then $\chi(\E^n)\le \chi_s(\E^n,\Lambda,\Lambda',\ell)$.

	Coulson~\cite{Coulson} used a weaker version of the following proposition to prove that $\chi(\mathbb{E}^{3})\leq 15$.

	\begin{proposition}\label{obs:coulson}
		Let $\Lambda$ be a full rank lattice in $\mathbb{E}^{n}$, $V$ be the Voronoi cell of $\Lambda$ about the origin and $R$ be the covering radius of $\Lambda$. If $\Lambda'$ is a sublattice of $\Lambda$ such that for any $v\in\Lambda'\setminus\{0\}$ we have $\dist(V,v+V)\ge 2R$, then $\chi(\mathbb{E}^{n})\leq |\Lambda/ \Lambda^{\prime}|$. 
	\end{proposition}

	\begin{proof}
		We need to show that $\chi_s(\E^n,\Lambda,\Lambda',2R)=|\Lambda/\Lambda'|$. Suppose to the contrary that according to $c(\Lambda,\Lambda')$ there exist monochromatic points $x$ and $y$ distance $2R$ apart. For each $v\in V$ we have that $v+\intr(V)$  is a subset of the ball $v+\ball(R)$ of diameter $2R$, so $x$ and $y$ belong to $v+\intr(V)$ and of $u+\intr(V)$, respectively, for some different $v,u\in \Lambda$. Then, according to the construction of $c(\Lambda,\Lambda')$, we must have $v-u\in \Lambda'$. We obtain a contradiction since $2R\le \dist(V,v-u+V)=\dist(v+V,u+V)<|x-y|=2R$, where the last inequality is strict since $x$ belongs to $v+\intr(V)$.  
	\end{proof}



	\begin{proposition}\label{prop:covering-packing2}
		If there exists a full rank lattice $\Lambda$ in $\E^n$ with covering-packing ratio not exceeding $2$, then $\chi(\E^n)\le 3^n$.
	\end{proposition}
	\begin{proof}
		Without loss of generality, assume that the packing radius of $\Lambda$ is $1$. Then the covering radius $R$ is at most $2$. Let $V$ be the Voronoi cell of $\Lambda$ around the origin. For any $v\in\Lambda\setminus\{0\}$, we have $2V\subset \{t\in\E^n:t\cdot v/|v|\le |v|\}$ (see~\eqref{eqn:voronoi cell prel}), therefore,
		\[
		\dist(V,3v+V)=\dist(2V,3v)\ge 3v \cdot v/|v| - |v|= 2|v|\ge 4\ge 2R.
		\]
		Now by \cref{obs:coulson} with $\Lambda'=3\Lambda$ we get $\chi(\E^n)\le 3^n$.
	\end{proof}
	\begin{remark}
		It is well-know that in any dimension there exist lattices with covering-packing ratio less than 3 (see \cite{Banaszczyk}, \cite{Rogers50}). In \cite{Henk}*{Theorem 5.2.3} it was proved that lattices with covering-packing ratio not exceeding $\sqrt{21}/2\approx 2.29$ exist in any dimension. By considering such lattices, one can follow the lines of the proof of Proposition~\ref{prop:covering-packing2} (with $\Lambda^\prime=4\Lambda$) and show that in any dimension $\chi(\E^n)\leq 4^n$. This bound could be potentially better in small dimensions than the current best asymptotic upper bound $(3+o(1))^n$ . 
	\end{remark}

	Let $\omega=e^{2\pi i/3}$ be a primitive cubic root of unity, $\EE=\Z[\omega]$ be the ring of Eisenstein integers. 
	Recall~\cite{ConSloane}*{p.~54} that a $\EE$-lattice (or Eisenstein lattice) is a set of the 
	form 
	\[\{\xi_1v_1+\dots+\xi_nv_n\;\colon\; 
	\xi_1,\dots,\xi_n\in\Z[\omega]\}\subset \mathbb{C}^n\,,\] 
	where $v_1,\dots,v_n$ are some linearly independent (over $\mathbb{C}$) vectors. 
	Any $\EE$-lattice can be regarded as a usual lattice of rank~$2n$ in $\E^{2n}$ that has an automorphism of order $3$ without nonzero fixed points, an automorphism given by the 
	multiplication by $\omega$. Note that multiplication by a $a+b\omega$, where $a,b\in \mathbb{R}$, acts on a 
	$\EE$-lattice $\Lambda$ as homothety-rotation with scaling factor 
	$|a+b\omega|=\sqrt{a^2-ab+b^2}$.
	
	\begin{theorem}\label{thm:7n}
		If there exists a $\EE$-lattice $\Lambda$ in $\mathbb{C}^n$ 
		with covering-packing ratio not exceeding $3/2$, then
		$\chi(\mathbb{E}^{2n}) \le 7^n$.
	\end{theorem}
	\begin{proof}
		Set $\Lambda^{\prime} = \alpha\Lambda$, where $\alpha=3+\omega$. Let $V$ be the Voronoi cell of $\Lambda$ around the origin. Assume that the packing radius of $\Lambda$ is $1$ and the covering radius of $\Lambda$ is $R\le 3/2$. For any $v\in \Lambda\setminus\{0\}$, with ``$\cdot$'' being the inner product in $\mathbb{C}^n$, we have
		\[
		\dist(K,\alpha v+K)=\dist(2K, \alpha v) \ge \alpha v \cdot v/|v|-|v|
		= |v|((\alpha v/|v|)\cdot(v/|v|)-1) 
		\ge \tfrac{3}{2}|v|\ge 3 \ge 2R. 
		\] 
		Since $\alpha=3+\omega$ has norm $7$ in $\EE$, we have $|\Lambda/ 
		\Lambda^{\prime}|=7^{n}$, 
		and $\chi(\mathbb{E}^{2n})\leq |\Lambda/ \Lambda^{\prime}|=7^{n}$  by \cref{obs:coulson}.
\end{proof}

\begin{remark}
	By using $2K\subset \{t\in\E^{2n}:t\cdot u\le |u|^2\}$ not only for $u=v$ but also for $u=-w^2v$, it is possible to prove the result of Theorem~\ref{thm:7n} under a slightly weaker assumption that the covering-packing ratio does not exceed $\sqrt{7/3}\approx 1.5275$. We are not aware of any lattices that would utilize this strengthening, so we have chosen to present the weaker but simpler statement.
\end{remark}

	\section{Computer assisted constructions for $n\in\{5,7,9\}$}\label{sec:computer}

	In this section we prove \cref{thm:An_star} using sublattice coloring schemes by explicitly providing the required lattices and sublattices. The verification is computer assisted and will be described below together with the required mathematical content.
	
	\subsection{Verification and constructions}

	Let $\Lambda$ be a lattice of rank $n$ in $\E^m$, i.e. $\Lambda = M \Z^n$ for a $m\times n$ generator matrix $M$ of rank $n$. Let $X=M\E^n$ be the subspace of $\E^m$ spanned by $\Lambda$, $V$ be the Voronoi cell of $\Lambda$ around the origin restricted to $X$ and $R$ be the covering radius of $\Lambda$. We can extend~\eqref{eqn:voronoi cell prel} as follows:
	\begin{equation}\label{eqn:voronoi cell}
		V = \bigcap_{z\in \Lambda\setminus\{0\}}\{x\in X:x\cdot z\le |z|^2/2\}
		= \bigcap_{z\in \Lambda\setminus\{0\},\, |z|\le 2R}\{x\in X:x\cdot z\le |z|^2/2\}.
	\end{equation} 


	We use certain symmetries available in the lattices under consideration. Suppose $H=\{x\in\E^m:x\cdot w=0\}$, $|w|=1$, is a hyperplane in $\E^m$ containing the origin. The reflection about $H$ is the isometry $R_w$ acting as $R_w(x)= x-2 (x\cdot w)w$. If a lattice is symmetric about a family of reflections, then the considerations can be reduced to a certain polyhedral convex cone $K$. The intersection of the Voronoi cell of the lattice {$\Lambda$} with the cone $K$ is a polyhedron which may have significantly fewer faces than the Voronoi cell {$V$} itself.
	\begin{lemma}\label{lem:K}
		Suppose $W$ is a finite family of unit vectors in $\E^m$ such that $R_{w}(\Lambda)=\Lambda$ for any $w\in W$. Let $G$ be the group of isometries of $X$ generated by the reflections $\{R_{w}:w\in W\}$. {Assume that the cone $K:=\{x\in X:x\cdot w\ge 0 \,\forall w\in W\}$ has non-empty relative interior, and let $V_1:=K\cap V$. Then,}\\ (i) for any $x\in X$ there exists $g\in G$ such that $g(x)\in K$.\\ 
		(ii) for any $x\in K$ we have $\dist(V,x+V)=2\, \dist(V_1,x/2)$; \\ (iii) $\displaystyle V_1
		= \bigcap_{z\in K\cap \Lambda\setminus\{0\},\, |z|\le 2R}\{x\in K:x\cdot z\le |z|^2/2\} $.
	\end{lemma}
	\begin{proof}
		We need the following facts, which are immediate by considering the squares of the required inequalities, {and noticing that $t\cdot w$ and  $(R_w(t)-t)\cdot w$ always have opposite signs:}\\ (a) If $|w|=1$, $z\cdot w\ge0$ and $t\cdot w\le 0$ (i.e. if $z$ and $t$ are on different sides of the hyperplane $\{x:x\cdot w\}=0$), then $|z-t|\ge |z-R_w(t)|$. \\ (b) Moreover, if $|w|=1$, $z\cdot w>0$ and $t\cdot w<0$, then $|z-t|< |z-R_w(t)|$.

		(i) {Observe that $G$ is finite. Indeed, all elements of $G$ are isometries of $X$ (and of $\Lambda$). For any $R'>0$, if $B(R')$ is the set of all bases with maximal length of basis vectors not exceeding $R'$, then $B(R')$ is an invariant set under $G$. Now, since $B(R')$ is finite and every isometry is uniquely determined by the image of some fixed base, we may conclude that $G$ is finite.}
		
		Let $z$ be in the relative interior of $K$, then $z\cdot w>0$ for any $w\in W$. Given $x\in {X}$, let $g_0\in G$ be such that $|z-g_0(x)|=\min \{|z-g(x)|:g\in G\}$. If $g_0(x)\in K$, we are done. Otherwise, there exists $w\in W$ such that $g_0(x)\cdot w<0$. Then by (b), $|z-R_w(g_0(x))|<|z-g_0(x)|$, a contradiction with the choice of $g_0$.

		(ii) Since $V$ is origin-symmetric and convex, $\dist(V,x+V)=\dist(2V,x)=2\dist(V,x/2)$. It remains to show that $\dist(V,x/2)=\dist(V_1,x/2)$. Since $V_1\subset V$, $\dist(V,x/2)\le \dist(V_1,x/2)$, so it remains to establish the converse inequality. 
		
		Suppose to the contrary $\dist(V,x/2)< \dist(V_1,x/2)$ and let $t\in V\setminus K$ be such that $\dist (V,x/2)=|x/2-t|$. 
		{Applying (i) to $t$ we obtain $g\in G$, such that $t'=g(t)\in K$. $\Lambda$ is invariant under $G$, and so is Voronoi cell $V$ of $\Lambda$, therefore $t'=g(t)\in V_1$. Representing $g=R_{w_1}\circ \dots\circ R_{w_k}$ with $w_j\in W$, and successively using (a), we obtain $|x/2-t|\ge |x/2-t'|$, contradiction.}  

		(iii) By~\eqref{eqn:voronoi cell}, clearly
		\[
		V_1 = \bigcap_{z\in \Lambda\setminus\{0\},\, |z|\le 2R}\{x\in K:x\cdot z\le |z|^2/2\}
		\subset 
		\bigcap_{z\in K\cap \Lambda\setminus\{0\},\, |z|\le 2R}\{x\in K:x\cdot z\le |z|^2/2\},
		\]
		so we only need to establish the converse inclusion.  
		
		Suppose to the contrary that for some $x\in K$ we have $|x|\le |z-x|$ for any $z\in K\cap \Lambda\setminus\{0\}$, $|z|\le 2R$, while for some $t\in \Lambda\setminus K$ we have $|x|>|t-x|$. Arguing as in part~(ii), we find $g\in G$ with $t'=g(t)\in K$ and  $|t-x|\ge|t'-x|$. Since $t'\in \Lambda$ we have $|x|\le |t'-x|$, and at the same time $|x|>|t-x|\ge|t'-x|$, a contradiction.
	\end{proof}

	Suppose the hypotheses of \cref{lem:K} hold. Let $F\subset \Lambda$ be the set of forbidden nodes of the lattice, i.e. those $x\in\Lambda\setminus \{0\}$ with $\dist (V,x+V)<2R$. Due to (ii) of \cref{lem:K}, it suffices to compute the set $F_1\subset K \cap \Lambda \setminus\{0\}$ of nodes $x$ satisfying $\dist(V_1,x/2)<R$ and then set $F:=\bigcup_{g\in G}g(F_1)$. 

	Now let us describe how we compute the distance from a point $y$ to a polytope $P$ in $\E^m$. Each $k$-dimensional face $f$ of $P$ is the intersection of a $k$-dimensional affine space $a(f)+A(f)\E^k$ with $P$, where $a(f)\in \E^m$ and $A(f)$ is the corresponding $m\times k$ matrix of rank $k$. If $A(f)^+$ denotes the Moore-Penrose inverse of $A(f)$ (see, e.g.~\cite{pseudoinverse}), then $B(f)=A(f)A(f)^{+}$ is the matrix of the projection operator on the range of $A(f)$, and so we have the following formula:
	\[
	\dist(y,P)=\min\{ |(I-B(f))(y-a(f))|: f\text{ is a face of }P\text{ and } a(f)+B(f)(y-a(f))\in f  \}.
	\]  
	If all vertices of $P$ have rational coordinates, then $A(f), A(F)^{+},$ and $B(f)$ have rational entries and can be computed precisely.


	It is easy to observe that any $x$ in $F_1$ or in a forbidden set $F$ has norm less than $4R$, which allows to restrict the initial candidates when $F_1$ is constructed. Generating $F$ from $F_1$ is usually straightforward as $G$ usually has a simple structure, e.g. contains all permutations of the coordinates.

	For a given sublattice $\Lambda'\subset \Lambda$, we would like to verify if $\Lambda'\cap F=\emptyset$, in which case $\chi(X)\le |\Lambda/\Lambda'|$ by \cref{obs:coulson}. This is performed using the following proposition. For convenience, we consider $M$ as a linear mapping from $\E^n$ to $X$ and denote by $M^{-1}$ its inverse. Set $\wt F:=M^{-1}F$ and fix a positive integer $s$. 
	\begin{proposition}\label{prop:sublattice check}
		Let $\Lambda'=MC\Z^n$ for a non-singular $n\times n$ matrix $C$ with integer entries. If $\Lambda' \cap F=\emptyset$, then necessarily:\\
		(i) $C\lambda\not\in \wt F$ for any $\lambda=(\lambda_1,\dots,\lambda_n)\in\Z^n$ with $\lambda_1+\dots+\lambda_n\ge0$ and $|\lambda_j|\le s$, $1\le j\le n$.\\
		Further, we have $\Lambda' \cap F=\emptyset$ if and only if either (ii) or (iii) holds, where:\\
		(ii) $C\lambda\not\in \wt F$ for any $\lambda\in\Z^n$ with $\lambda_1+\dots+\lambda_n\ge0$ and $|\lambda_j|\le m_j \gamma$, $1\le j\le n$, where $m_j$ is the (Euclidean) norm of the $j$-th row of $C^{-1}$ and $\gamma=\max\{|f|:f\in\wt F\}$;\\
		(iii) $C^{-1}f\not\in \Z^n$ for any $f=(f_1,\dots,f_n)\in\wt F$ with $f_1+\dots+f_n\ge0$.
	\end{proposition}
	\begin{proof}
		Necessity of (i) is obvious. For (ii) and (iii), origin-symmetry of $\Lambda'$, $F$ and $\wt F$ allows to impose the conditions $\lambda_1+\dots+\lambda_n\ge 0$ and $f_1+\dots+f_n\ge 0$. Equivalence of $\Lambda' \cap F=\emptyset$ and (iii) is now clear. For (ii), it is enough to observe that if $v_j$ is the $j$-th row of $C^{-1}$ and $f\in\wt F$, then  $\lambda_j=v_j\cdot f$ and by the Cauchy-Schwartz inequality $|\lambda_j |\le m_j |f|$.
	\end{proof}

	Informally, if $\Lambda'$ has a vector in $F$, it is likely to have ``small'' coefficients of its representation, so it makes sense to begin verification with (i). We found that choosing $s=2$ for $n=5$ worked well, while for larger $n$ we usually selected $s=1$. After (i) is verified and no vectors in $F$ is found, we choose to proceed with either (ii) or (iii) depending on the numbers of required evaluations of $C\lambda$ or of $C^{-1}f$.

%

\begin{proof}[Proof of \cref{thm:An_star}]
	We will apply \cref{obs:coulson}. 
	
	Suppose first that $n\in\{5,9\}$. We consider a dilation of the $A_n^*$ lattice, see~\cite{ConSloane}*{Sect.~4.6.6, p.~115}. Namely, set $\Lambda=M\Z^n\subset \E^{n+1}$, where
	\[
	M =  	\left(\begin{array}{rrrr}
		-n & 1 & \dots & 1  \\
		1 & -n & \dots & 1 \\
		&&\dots& \\
		1 & 1 & \dots & -n \\
		1 & 1 & \dots & 1 
	\end{array}\right).
	\]
	The matrices $C_n$ generating required sublattices by $\Lambda'=M C_n \Z^n$ are given as follows: \begin{gather*}
	C_5 = \left(\begin{array}{rrrrr}
		-2 & 1 & -2 & -1 & 0 \\
		-3 & 1 & 0 & -1 & -2 \\
		0 & 1 & 1 & -1 & -3 \\
		-2 & 0 & -2 & 2 & -2 \\
		-2 & -2 & 0 & 0 & -2
	\end{array}\right), \\
	C_9 = \left(\begin{array}{rrrrrrrrr}
		0 & 0 & -3 & 1 & 0 & 0 & -1 & 1 & 0 \\
		1 & 0 & -3 & 1 & 1 & 0 & -1 & 4 & 1 \\
		0 & 0 & -2 & 1 & 0 & -1 & -1 & 1 & 3 \\
		0 & 0 & -3 & 4 & 0 & 0 & -1 & 1 & 0 \\
		0 & 3 & -3 & 1 & 0 & 0 & -1 & 1 & 0 \\
		3 & 0 & -3 & 1 & 0 & 0 & 2 & 1 & 0 \\
		0 & 0 & -4 & 2 & 0 & 3 & -1 & 2 & 0 \\
		0 & 0 & -3 & 1 & 3 & 0 & -1 & 1 & 0 \\
		-1 & 0 & -3 & 1 & -1 & 1 & 1 & 1 & -1
	\end{array}\right).
 	\end{gather*}
 The verification of the hypothesis of \cref{obs:coulson} was performed on a computer using \cref{lem:K} and \cref{prop:sublattice check}~(iii), see the ``59dimAnstar'' SageMath script at \cite{ABPR-github}.
 
 For $n=7$, we use the $E_7^*$ lattice, see~\cite{ConSloane}*{Sect.~4.8.2, p.~125}. Set $\Lambda=M\Z^7\subset\E^8$, where
 	\[
 M = \left(\begin{array}{rrrrrrr}
 	-1 & 0 & 0 & 0 & 0 & 0 & -\tfrac{3}{4} \\
 	1 & -1 & 0 & 0 & 0 & 0 & -\tfrac{3}{4} \\
 	0 & 1 & -1 & 0 & 0 & 0 & \tfrac{1}{4} \\
 	0 & 0 & 1 & -1 & 0 & 0 & \tfrac{1}{4} \\
 	0 & 0 & 0 & 1 & -1 & 0 & \tfrac{1}{4} \\
 	0 & 0 & 0 & 0 & 1 & -1 & \tfrac{1}{4} \\
 	0 & 0 & 0 & 0 & 0 & 1 & \tfrac{1}{4} \\
 	0 & 0 & 0 & 0 & 0 & 0 & \tfrac{1}{4}
 \end{array}\right).
 \]
 A required sublattice is $\Lambda'=M C_7 \Z^7$, where
 \[
 C_7 = \left(\begin{array}{rrrrrrr}
 	0 & -4 & -5 & -3 & -4 & -4 & -1 \\
 	-1 & -5 & -10 & -7 & -5 & -5 & -4 \\
 	-2 & -2 & -9 & -4 & -5 & -4 & -4 \\
 	-3 & -2 & -5 & -4 & -4 & -1 & -3 \\
 	-1 & -1 & -4 & -1 & -3 & 0 & -3 \\
 	-2 & 0 & -1 & 0 & 0 & 0 & 0 \\
 	0 & 4 & 6 & 4 & 4 & 4 & 4
 \end{array}\right).
 \]
 The verification of the hypothesis of \cref{obs:coulson} was performed on a computer using \cref{lem:K} and \cref{prop:sublattice check}~(iii), see the ``7dimE7star'' SageMath script at \cite{ABPR-github}.
\end{proof}

\subsection{Sublattice search strategies}

For a given lattice, when the dimension is relatively small like $n=5$, it is computationally feasible to check all possible sublattices (of certain index) exhaustively, see \cref{prop:negative dim 4}. If the dimension is larger or we are interested in getting a result quickly, then we used a combination of the following two approaches. 

{\bf Randomized search among short non-forbidden nodes.} The basis vectors of the desired sublattice are not elements of the forbidden set $F$. Since our goal is minimizing the index of the sublattice, which equals to the volume of the fundamental parallelepiped of the sublattice, it is natural to search for the basis among {\it short} vectors. To this end, we sort the non-zero lattice nodes which are not in $F$ in ascending order by length. Let $G$ be the set of $N$ smallest such non-forbidden nodes. Then we randomly and uniformly draw $n$ samples from $G$ to form a basis for the sublattice and check as outlined in \cref{prop:sublattice check} whether the resulting sublattice is suitable. The choice of the parameter $N$ is performed in experimental manner: if $N$ is too small, there will not be a good sublattice basis among elements of $G$; while if $N$ is too large, then it may take too much time for the random samples to find a good sublattice.

{\bf ``Gradient'' descent.} Once we found a good sublattice $\Lambda'$, we take one of the vectors generating it and substitute it with other short non-forbidden nodes, choosing one that minimizes the index of the resulting $\Lambda'$ while satisfying $\Lambda'\cap F=\emptyset$. This is repeated for all vectors generating $\Lambda'$ (possibly multiple times) until no further improvement of the index is possible.

\subsection{Lower bounds in small dimensions}


A general lower bound by Coulson~\cite{Coulson}*{Th.~4.5} is $\chi_s(\E^n,\Lambda,\Lambda',\ell)\ge 2^{n+1}-1$, so one cannot improve the inequality $\chi(\E^3)\le 15$ using a sublattice coloring. The lower bounds on sublattice colorings we give below are for specific lattices $\Lambda$ and only for the case when the excluded distance $\ell$ is twice the covering radius of the lattice.

Existence of linear coloring using $A_4^*$ lattice yielding $\chi(\E^4)\le 54$ was stated in~\cite{RadToth}. A better bound $\chi(\E^4)\le 49$ achieved by a sublattice coloring was stated in~\cite{CoulsonPayne}. We prove this bound in \cref{thm:7n} using the $D_4$ lattice (see~\cite{ConSloane}*{p.~118}). We verified that one cannot do better, as well as obtained similar analysis for $A_4^*$ and $A_5^*$.
\begin{proposition}\label{prop:negative dim 4}
	If $R$ is the covering radius of a lattice $\Lambda$, then the smallest values of sublattice chromatic numbers (over all possible sublattices $\Lambda'$) are as follows:
	\begin{center}
		\begin{tabular}{|c|c|c|}
			\hline
			$n$ & $\Lambda$ & $\displaystyle\min_{\Lambda'}\chi_s(\E^n,\Lambda,\Lambda',2R)$\\
			\hline
			$4$ & $D_4$ & $49$ \\
			\hline
			$4$ & $A_4^*$ & $54$ \\
			\hline
			$5$ & $A_5^*$ & $140$ \\
			\hline
		\end{tabular}
	\end{center}
\end{proposition}
\begin{proof}
	The proof is by computer search. For sublattice colorings, see the scripts ``4dim exh sublattice ...'' and ``5dim exh sublattice A5star'' at \cite{ABPR-github}, with the outputs for each possible index value for the $5$-dimensional case given in the ``exh 5 results'' folder at \cite{ABPR-github}. The main idea is that all possible sublattices of a given lattice can be obtained by $\Lambda'=C\Lambda$ where an integer matrix $C$ is in the Hermite normal form~\cite{HermiteNF}. All such matrices $C$ of given determinant can be generated and the resulting sublattices tested using \cref{lem:K} and \cref{prop:sublattice check}. We use a slightly modified version of the Hermite normal form where the entries above a pivot $l$ are from the set $-\lfloor\tfrac l2\rfloor+\{0,\dots,l-1\}$ instead of $\{0,\dots,l-1\}$. This increases the chances that there is a ``small'' linear combination of the vectors from the sublattice basis (see~(i) of \cref{prop:sublattice check}) which belongs to the forbidden nodes, and speeds up the computations. 
\end{proof}

	\section{Constructions using covering-packing ratio}\label{sec:covering-packing}

	In this section we will prove \cref{thm:3n,thm:evendims} using \cref{prop:covering-packing2,thm:7n}, respectively, by describing certain appropriate lattices.



	\subsection{Laminated lattices and proof of \cref{thm:3n}}


	A laminated lattice $\Lambda_n$ is a full rank lattice in $\mathbb{E}^n$ that is obtained recursively in the following way. $\Lambda_1=2\mathbb{Z}$, and a laminated lattice $\Lambda_{n+1}$ is a full rank lattice in $\mathbb{E}^{n+1}$ containing some laminated $\Lambda_n$, such that packing radius of $\Lambda_{n+1}$ is equal to $1$ and the volume of Voronoi cell of $\Lambda_{n+1}$ is minimal possible. 

	Notice that laminated lattice $\Lambda_{n}$ may not be unique, so $\Lambda_{n}$ formally is a collection of lattices. For example for $n\leq 24$ laminated lattices $\Lambda_n$ are unique, except for $n=11$ (2 lattices $\Lambda_{11}^{\min}, \Lambda_{11}^{\max}$), $n=12$ ($\Lambda_{12}^{\min}, \Lambda_{12}^{\text{mid}}, \Lambda_{12}^{\max}$) and $n=13$ ($\Lambda_{13}^{\min}, \Lambda_{13}^{\text{mid}}, \Lambda_{13}^{\max}$), see~\cite[Figure 6.1]{ConSloane}. Additionally, since packing radius of any $\Lambda_{n}$ is equal to $1$, the covering-packing ratio of any $\Lambda_{n}$ coincides with the covering radius. A comprehensive exposition on laminated lattices can be found in~\cite{ConSloane}*{Ch.~6}.

			\begin{table}
	\begin{center}
		\begin{tabular}{| c | c | c | c |}
			\hline
			$n$ & covering-packing ratio & $n$ & covering-packing ratio \\
			\hline
			$9$ & $\sqrt{5/2}\approx 1.581138$ & $24$ & $\sqrt{2}\approx 1.414213$\\
			$10$ & $\sqrt{8/3}\approx 1.632993$ & $25$ & $\sqrt{5/2}\approx 1.581138$\\
			$11$ & $\sqrt{3}\approx 1.732050$ & $26$ & $\sqrt{8/3}\approx 1.632993$\\
			$12$ & $\sqrt{3}\approx 1.732050$ & $27$ & $\sqrt{3}\approx 1.732050$\\
			$13$ & $\sqrt{13}/2\approx 1.802776$ & $28$ & $\sqrt{3}\approx 1.732050$\\
			\blue{$14$} & \blue{$\leq \sqrt{55}/4\approx 1.854050$} & $29$ & $\sqrt{13}/2\approx 1.802776$ \\
			\blue{$15$} & \blue{$\leq \sqrt{173/48} \approx 1.898465$} & \blue{$30$} & \blue{$\leq \sqrt{55}/4\approx 1.854050$}\\
			$16$ & $\sqrt{3}\approx 1.732050$ & \blue{$31$} & \blue{$\leq \sqrt{173/48} \approx 1.898465$}\\
			$17$ & $\sqrt{13}/2\approx 1.802776$ & $32$ & $\sqrt{3}\approx 1.732050$\\
			\blue{$18$} & \blue{$\leq \sqrt{55}/4 \approx 1.854050$} & $33$ & $\sqrt{13}/2\approx 1.802776$\\
			\blue{$19$} & \blue{$\leq \sqrt{173/48} \approx 1.898465$} & \blue{$34$} & \blue{$\leq \sqrt{55}/4\approx 1.854050$}\\
			\blue{$20$} & \blue{$\leq \sqrt{179/48} \approx 1.931106$} & \blue{$35$} & \blue{$\leq \sqrt{173/48} \approx 1.898465$} \\
			\blue{$21$} & \blue{$\leq \sqrt{185/48} \approx 1.963204$} &\blue{$36$} & \blue{$\leq \sqrt{179/48} \approx 1.931106$} \\
			\blue{$22$} & \blue{$\leq \sqrt{379/96} \approx 1.986937$} & \blue{$37$} & \blue{$\leq \sqrt{185/48} \approx 1.963204$}\\
			\blue{$23$} & \blue{$\leq 1.936501$} & \blue{$38$} & \blue{$\leq \sqrt{379/96} \approx 1.986937$}  \\
			\hline
		\end{tabular}
	\end{center}
	\caption{Covering-packing ratios of laminated lattices $\Lambda_n$. In the dimensions where $\Lambda_{n}$ is not unique the maximum possible covering ratio is listed. If covering radius is not known, an upper bound on the covering ratio is listed (blue values). }
	\label{tab:sqrt3}
\end{table}

	\begin{proposition}\label{prop:mid dim cp 2}
		The covering-packing ratios of laminated lattices in dimensions $9\leq n\leq 38$ satisfy the equalities or the upper bounds given in Table~\ref{tab:sqrt3}.
	\end{proposition}
	\begin{proof}
		Let $\rho_n$ denote the greatest covering radius of any laminated lattice $\Lambda_n$. By \cite{ConSloane}*{Th.~1, p.~164}, we obtain the exact values of $\rho_n$ for $n\in\{9,10,11,12,16,24,25,26,27,28,32\}$ as listed in \cref{tab:sqrt3}. More precisely, these values are equal to the so-called subcovering radius $h_n$ from \cite{ConSloane}*{Table~6.1}   provided $h_n\le\sqrt{3}$ (see \cite{ConSloane}*{Th.~1, p.~164} for more details).
		
		Next we note (see~\cite{ConSloane}*{p.~163}) that each $\Lambda_n$ can be represented in $\E^n$ as $\cup_{j\in\Z}\Lambda_{n-1}^{(j)}$, where $\Lambda_{n-1}^{(j)}$ is a translate of a certain laminated lattice such that any point of $\Lambda_{n-1}^{(j)}$ has $n$-th coordinate equal to $j\sqrt{\pi_{n-1}}$. The values of $\pi_{n-1}$ are given in \cite{ConSloane}*{Table~6.1}. Now for any point $x=(x_1,\dots,x_n)\in \E^n$ we can find $j$ minimizing $|x_n-j\sqrt{\pi_{n-1}}|$ and then a lattice node $y\in\Lambda_{n-1}^{(j)}$ with $|(x_1,\dots,x_{n-1},j\sqrt{\pi_{n-1}})-y|\le \rho_{n-1}$. This yields
		\begin{equation}\label{ineq:rhon}
		\rho_{n}\leq \sqrt{\frac{\pi_{n-1}}{4}+\rho_{n-1}^{2}}.
		\end{equation}
		Starting with the already obtained values of $\rho_n$ for $n\in\{12,16,28,32\}$, consecutive applications of \eqref{ineq:rhon} imply the upper bounds on $\rho_n$ for all the remaining values of $n$ in \cref{tab:sqrt3} except for $n=23$. The equalities for $n\in\{13,17,29,33\}$ follow from the lower bound $h_n\le\rho_n$ and the values of $h_n$ given in \cite{ConSloane}*{Table~6.1}. Finally, the inequality for $n=23$ is a consequence of the computational result~\cite{SikSchVall}*{Table~2}.
	\end{proof}
	\begin{remark}
	In~\cite{SikSchVall}, numerical values of $\rho_n$ appear to match $h_n$ for $n\in\{14,15,18,23\}$, which suggests that possibly $\rho_n=h_n$ in these cases. However, it is not clear from~\cite{SikSchVall} whether the computations were performed using exact arithmetic and the results were converted to numerical ones only at the end. Also, according to~\cite{SikSchVall}, the covering radius $\rho_{13}^{\text{mid}}$ of $\Lambda_{13}^{\text{mid}}$ satisfies $\rho_{13}^{\text{mid}}\approx\sqrt{3}<\rho_{13}=\sqrt{13}/2$. However, $\Lambda_{13}^{\text{mid}}$ is not extendable to $\Lambda_{14}$ (see~\cite{ConSloane}*{Figure~6.1}), so~\eqref{ineq:rhon} is not applicable with $n=14$ and $\rho_{13}$ replaced by $\rho_{13}^{\text{mid}}$.
	\end{remark}
	
	Laminated lattices provide a recursive construction that allows to obtain lattices with a small covering-packing ratio. We now will consider how the covering-packing ratio behaves when we take a sum of lattices. 
	
	First, one can consider a direct sum $\Lambda_1\oplus \Lambda_2$ of two lattices $\Lambda_1$ and $\Lambda_2$ with covering-packing ratios $\rho_1$ and $\rho_2$ respectively. It can be easily seen that the covering-packing ratio of $\Lambda_1\oplus \Lambda_2$ does not exceed $\sqrt{\rho_1^2+\rho_2^2}$. For instance, by using the Leech lattice $\Lambda_{24}$, we obtain a lattice $\Lambda_{24}\oplus \Lambda_{24}$ with covering-packing ratio at most $2$ in dimension $n=48$.
	
	Other direction would be to consider a $\pi/3$-sum of two lattices, instead of the direct sum. We summarize  this approach in the following proposition.
	
	\begin{proposition}\label{prop:60degsum}
		For $i=1,2$ let $\Lambda_i$ be a full rank lattice in $\E^{n_i}$ with a covering-packing ratio $\rho_i$. Provided $n_1\geq n_2$, there exists a full rank lattice $\Lambda$ in $\E^{n_1+n_2}$ with the covering-packing ratio not exceeding $\sqrt{\rho_{1}^{2}+\frac{3}{4}\rho_{2}^{2}}$.
	\end{proposition}
	
	\begin{proof}
		We write $({x},{y})\in \E^{n_1+n_2}$ with ${x}\in \E^{n_1}$ and ${y}\in \E^{n_2}$. Define a lift operation $L:\E^{n_2}\to \E^{n_1}$ by $f({y})=({y}, 0,\ldots,0)$. 
		
		Assume that packing radii of lattices $\Lambda_1$ and $\Lambda_2$ are equal to $1$. Let $\Lambda$ be a $\frac{\pi}{3}$-sum of lattices $\Lambda_1$ and $\Lambda_2$ defined by $$\Lambda=\left\{\left({x}+\frac{L{y}}{2}, \frac{\sqrt{3}}{2}{y}\right) \; : \; {x}\in \Lambda_1, {y} \in \Lambda_2\right\}.$$
		
		It is easy to see that $\Lambda$ is a full rank lattice in $\E^{n_1+n_2}$. 
		
		Now we show that the packing radius of $\Lambda$ is equal to $1$. It is enough to show that any nonzero vector ${v}=({x}+\frac{L{y}}{2}, \frac{\sqrt{3}}{2}{y})$ of $\Lambda$ has length at least 2. We have
		$$\left|{v}\right|^2=\left|{x}+\frac{L{y}}{2}\right|^2+\left|\frac{\sqrt{3}}{2}{y}\right|^2=\left|{x}\right|^2+{x}\cdot L{y}+\left|{y}\right|^2\geq \left|{x}\right|^2-\left|{x}\right|\left|{y}\right|+\left|{y}\right|^2.$$
		Now, if one of ${x}$  or ${{y}}$ iz a zero vector, then clearly $|{{v}}|^2\geq 4$. If both ${x}$ and ${y}$ are non-zero, then $|{x}|\geq 2$, $|{y}|\geq 2$, and so 
		$$\left|{{v}}\right|^2\geq \frac{1}{2}\left|{x}\right|^2+\frac{1}{2}\left|{y}\right|^2\geq 4.$$
		
		To estimate the covering radius of $\Lambda$, let $(\hat{x}, \frac{\sqrt{3}}{2} \hat{y})$ be an arbitrary point of $\E^{n_1+n_2}$.	Since covering radius of $\Lambda_2$ is $\rho_2$, let ${y}\in \Lambda_2$ be such that $|\hat{y}-{y}|\leq \rho_2$. Similarly let ${x}\in \Lambda_1$ be such that $\left|\hat{x}-{x}-\frac{L{y}}{2}\right|\leq \rho_1$. Then $({x}+\frac{L{y}}{2}, \frac{\sqrt{3}}{2}{y})\in \Lambda$ and 
		$$\left|(\hat{x}, \frac{\sqrt{3}}{2}\hat{y})-({x}+\frac{L{y}}{2}, \frac{\sqrt{3}}{2}{y})\right|\leq \sqrt{\rho_1^{2}+\frac{3}{4}\rho_2^2}.$$
	\end{proof}

\begin{proof}[Proof of \cref{thm:3n}.]
	Due to \cref{prop:covering-packing2}, it suffices to show existence of a lattice with covering-packing ratio at most $2$ for each required dimension $n$.
	
	For $n\le 8$ the lattices with the smallest known covering-packing ratio are listed in~\cite{Sch-Va}*{Table~3, Table~4}; the ratii are strictly smaller than $2$. For $9\le n\le 38$, we use \cref{prop:mid dim cp 2}. 
	
	For $n=48$, by Proposition~\ref{prop:60degsum} the $\frac{\pi}{3}$-sum $\Lambda=\Lambda_{24}\oplus_{\pi/3}\Lambda_{24}$ of two Leech lattices has the covering-packing ratio at most $\sqrt{\frac{7}{2}}$. (Note that $\Lambda=\Lambda_{24}\oplus\Lambda_{24}$ can also be used as it has the covering-packing ratio at most $2$.)
	
	Finally, for $n=49$,  by Proposition~\ref{prop:60degsum} the $\frac{\pi}{3}$-sum $\Lambda=\Lambda_{25}\oplus_{\pi/3}\Lambda_{24}$ of two laminated lattices has the covering-packing ratio not exceeding $2$.	  
\end{proof}

	\subsection{Eisenstein lattices and proof of \cref{thm:evendims}}
	
	\begin{proof}[Proof of \cref{thm:evendims}.]
		By \cref{thm:7n}, for each dimension in question we only need to find an Eisenstein lattice with covering-packing ratio at most $3/2$. We list suitable lattices in the following table.
			\begin{center}
			\begin{tabular}{|c|c|c|c|}
				\hline
				$n$ & lattice & covering-packing ratio & reference \\
				\hline
				$2$ & $A_2$ & $2/\sqrt{3}$ & \cite{ConSloane}*{p.~110} \\
				$4$ & $D_4$ & $\sqrt{2}$ &  \cite{ConSloane}*{p.~119} \\
				$6$ & $E_6^*$ & $\sqrt{2}$ &  \cite{ConSloane}*{p.~127} \\
				$8$ & $E_8$ & $\sqrt{2}$ &  \cite{ConSloane}*{p.~121, p.~161} \\
				$24$ & $\Lambda_{24}$ & $\sqrt{2}$ &  \cite{ConSloane}*{p.~161} \\
				\hline			
			\end{tabular}
		\end{center}
	\end{proof}

%
%
%

	\section{Open questions}\label{sec:questions}

		\begin{openquestion}
		Do there exist better sublattice coloring schemes (perhaps for other lattices) than those in \cref{thm:An_star,thm:evendims} for $4\le n\le 9$?
	\end{openquestion}
	
	We believe that the answer is negative for $n=4,5$, i.e. $\min \chi_s(\E^4,\Lambda,\Lambda',\ell)=49$ and $\min \chi_s(\E^5,\Lambda,\Lambda',\ell)=140$, where the minima are taken over all full rank lattices $\Lambda$, sublattices $\Lambda'$ and all $\ell>0$.
	
	\begin{openquestion}
		Is it possible to extend the result of \cref{thm:evendims} to other dimensions? In particular, does there exist an Eisenstein lattice with covering-packing ratio at most $3/2$ in any other dimension than listed in~\cref{thm:evendims}?
	\end{openquestion}
	\begin{openquestion}
		Is it possible to obtain a new upper bound on $\chi(\E^n)$ for some $n$ using a modification of the technique of \cref{thm:7n}, perhaps using Hurwitz quaternionic integers~\cite{ConSloane}*{Sect.2.2.6, p.~53}?
	\end{openquestion}

	Everywhere below in this section $R$ denotes the covering radius of a lattice $\Lambda$.
	
	\begin{openquestion}
		Is it true that for every $n$ there exists a full rank lattice $\Lambda$ and a sublattice $\Lambda'\subset\Lambda$ such that $\chi_s(\E^n,\Lambda,\Lambda',2R)<3^n$?
	\end{openquestion}

	An affirmative answer immediately implies $\chi(\E^n)<3^n$. The results in this paper give an affirmative answer for $n\le 9$.

	\begin{openquestion}\label{open:cov-pack}
		Find all dimensions $n$ for which there exists a full rank lattice in $\E^n$ with covering/packing ratio at most $2$.
	\end{openquestion}

	For any such dimension we immediately have $\chi(\E^n)\le\chi_s(\E^n,\Lambda,3\Lambda,2R)= 3^n$ by \cref{prop:covering-packing2}. It is believed (see~\cite[Problem 4.1]{Sch-Va}, \cite[Problem 1.1, 1.2]{Zong}, \cite[Question 3.2]{Muraz}, \cite[p.63 Problem 6]{RPDG}) that for a sufficiently large $n$ any lattice in $\E^n$ has covering/packing ratio {\it at least} $2$, i.e. there are non-lattice packings of $\E^n$ which are denser than lattice packings. Notice that discrete sets with covering-packing ratio at most $2$ having non-lattice structure can be easily constructed in any dimension using maximal separated sets in a torus. On the other hand, existence of a lattice $\Lambda$ in $\E^n$ with a covering-packing ratio $2+o(1)$ as $n\to\infty$ was established by Butler~\cite{Butler}.

	\begin{openquestion}
	Find all $n$ such that there exists a laminated lattice $\Lambda_n$ with covering radius at most $2$.
\end{openquestion}	
We know that all dimensions $n$ up to $38$ satisfy this property. The laminated lattices construction is based on maximizing the density of the packing, while in our context the covering radius is of interest. 
\begin{openquestion}
	Is it possible to modify the construction of laminated lattices to give an affirmative answer to \cref{open:cov-pack} for some new values of $n$?
\end{openquestion}

%
	
	
	Although sublattice colorings allow to prove upper bounds on $\chi(\E^n)$ better than $3^{n}$ in many dimensions, study of their limitations is of separate interest.
	\begin{openquestion}
		Is it possible to improve the Coulson's bound $\chi_s(\E^n,\Lambda,\Lambda',\ell)\ge 2^{n+1}-1$ from~\cite{Coulson}, at least for some $n$, or for the case $\ell=2R$, or allowing the dependence of the bound on some additional characteristics of the lattice?
	\end{openquestion}

	\begin{openquestion}
		Does the inequality $\min\chi_s(\E^n,\Lambda,\Lambda',2R)\le \min \{\chi_s(\E_n,\Lambda,\Lambda',\ell):\ell>2R\}$ hold for any $n$, where the minima are taken over all full-rank lattices $\Lambda$ and sublattices $\Lambda'$? In other words, can we consider only the case $\ell=2R$ when searching for the smallest sublattice coloring chromatic numbers?
	\end{openquestion}
	We observe that a suitable sublattice not only cannot contain nodes from the forbidden set, but also it cannot contain any nodes whose integer multiple belongs to the forbidden set. Therefore, we are inclined to believe that the answer to the previous question is affirmative. 
			
	\begin{openquestion}
		What are other coloring approaches apart from linear or sublattice coloring schemes which can yield (possibly with computer assistance) new upper bounds on $\chi(\E^n)$ for small $n$?
	\end{openquestion}

%
%
%
%
%
%
%
%
%
%

\end{document}